\documentclass[paper=a4, fontsize=11pt]{scrartcl} 

\usepackage[T1]{fontenc} 
\usepackage{fourier} 
\usepackage[english]{babel} 
\usepackage{amsmath,amsfonts,amsthm} 
\usepackage{amsmath, calligra, mathrsfs}
\usepackage{amssymb}
\usepackage{mathtools}
\usepackage{hyperref}
\usepackage{mathdots}
\usepackage{array}
\usepackage[mathscr]{euscript}
\usepackage{verbatim}
\usepackage[dvipsnames]{xcolor}
\usepackage{mdframed} 
\usepackage{soulutf8}
\usepackage{stmaryrd}
\usepackage{tikz-cd}
\usepackage{hyperref}
\usepackage{lipsum} 

\usepackage{sectsty} 
\allsectionsfont{\centering \normalfont\scshape} 

\usepackage{fancyhdr} 
\usepackage{xurl}
\newcommand*{\sheafhom}{\mathcal{H}\kern -.5pt om}
\numberwithin{equation}{section} 
\numberwithin{figure}{section} 
\numberwithin{table}{section} 

\newtheorem{thm}{Theorem}[section]
\newtheorem{cor}[thm]{Corollary}
\newtheorem{prop}[thm]{Proposition}
\newtheorem{obs}[thm]{Observation}
\newtheorem{lem}[thm]{Lemma}

\theoremstyle{definition}
\newtheorem{defn}[thm]{Definition}

\theoremstyle{remark}
\newtheorem{rem}[thm]{Remark}

\DeclareMathOperator{\Sym}{Sym}

\DeclareMathOperator{\Rep}{Rep}

\newcommand{\horrule}[1]{\rule{\linewidth}{#1}} 

\title{	
	\normalfont \normalsize 
	\textsc{} \\ [25pt] 
	\horrule{0.5pt} \\[0.4cm] 
	\huge Gamma vectors as inverted Chebyshev expansions, type A to B transformations, and connections to algebraic structures

	\horrule{2pt} \\[0.5cm] 
}

\author{Soohyun Park} 

\date{\normalsize November 3, 2025} 

\begin{document}
	
	\maketitle

	\begin{abstract}
		\noindent Given a reciprocal/palindromic polynomial of even degree, we show that the gamma vector is essentially given by an inverted Chebyshev polynomial basis expansion. As an immediate consequence, we characterize real-rootedness of a linear combination of Chebyshev polynomials in terms of real-rootedness of that of the reciprocal polynomial built out of an inverted scaled tuple of the coefficients with one fixed and the rest divided by 2. It can be taken as a counterpart for arbitrary dimensions of a recent result of Bel-Afia--Meroni--Telen on hyperbolicity of Chebyshev curves with respect to the origin. In general, Chebyshev varieties serve as a counterpart of toric varieties in sparse polynomial root finding. Apart from this, the inverted Chebyshev expansion also yields connections between intrinsic properties of the gamma vector construction and the geometric combinatorics of simplicial complexes and posets. \\
		
		\noindent We find this by applying work of Hetyei on Tchebyshev subdivisions and Tchebyshev posets. In particular, we find that the gamma vector transformation is closely related to $f$-vectors of simplicial complexes resulting from successive edge subdivisions that transform the type A Coxeter complex to the type B Coxeter complex. Lifting to this to the level of $ce$-indices (a modification of $cd$-indices), we show that the gamma vector inverted Chebyshev polynomial expansion lifts to a sum of $ce$-indices of cross polytope triangulations which can be computed using descent statistics involving edge labelings of maximal chains. While there are many examples in the literature where gamma positivity involving descent statistics, it is interesting to note that we have this without initial structural assumptions on the input polynomial apart from being reciprocal/palindromic. Finally, we note that all of these can be repeated with Chebyshev polynomials of the second kind after taking derivatives. This gives connections to Hopf algebras and quasisymmetric functions along with Lefschetz-type maps induced by $\mathfrak{sl}_2(\mathbb{C})$-representations.
	\end{abstract}
	
	\section*{Introduction} 

	Our main object of study is the gamma vector, which occurs in many different contexts including permutation statistics, Euler characteristics of nonpositively curved (piecewise Euclidean) manifolds, and triangulations of (boundaries of) polytopes (see \cite{Athgam} for a survey). \\
	
	Given a polynomial $h(t)$ of degree $d$, the function $\gamma(t)$ is defined (Proposition 2.1.1 on p. 272 of \cite{Gal}) using the relation \[ h(t) = (1 + t)^d h \left( \frac{t}{(1 + t)^2} \right). \]
	
	We will focus on the case where $h(t)$ and of even degree (which is the case in most existing applications and how it was originally defined). Then, we can set the $\gamma_i$ to be the unique coefficients such that \[ h(t) = \sum_{i = 0}^{ \frac{d}{2}  }  \gamma_i x^i (1 + x)^{d - 2i}. \]
	
	Note that positivity of the gamma vector lies somewhere between unimodality and real-rootedness when $h(t)$ has nonnegative coefficients in addition to being reciprocal (having $h_k = h_{d - k}$) and unimodal (Observation 4.2 on p. 82 of \cite{Pet} and Remark 3.1.1 on p. 277 of \cite{Gal}) and was originally motivated by real-rootedness considerations. \\

	The perspective we take here on gamma vectors of reciprocal polynomials comes from combining this expansion often used in applications and an explicit formula  for the gamma vector we had considered earlier \cite{Pgam} by taking compositional inverses (with $C(u)$ being the generating function of the Catalan numbers and $\widetilde{C}(u) \coloneq C(u) - 1$). More specifically, we consdered a modification of the gamma vector for basis elements $x^k + x^{d - k}$ of reciprocal polynomials of degree $\le d$ for even $d$ where the products of such vectors with $h_k$ add up to the gamma vector of the initial polynomial $h(t)$ being studied (Proposition \ref{modgambasis}). Continuing with this train of thought leads to an expression for the gamma polynomial which is a sort of inverted Chebyshev basis expansion (Theorem \ref{gamchebinv}). As an immediate consequence, we find a characterization of real-rootedness of linear combinations of Chebyshev polynomials in terms of that of the reciprocal polynomial built out of an inverted scaled tuple of the coefficients with one fixed and the rest divided by 2 (Corollary \ref{cheblinreal}) since the gamma polynomial $\gamma(t)$ of a reciprocal polynomials $h(t)$ is real-rooted if and only if $h(t)$ is. This is a counterpart of a recent result of Bel-Afia--Meroni--Telen (Theorem 2.5 on p. 7 of \cite{BAMT}) on the hyperbolicity of Chebyshev curves with respect to the origin. In general, Chebyshev varieties serve as a counterpart of toric varieties in sparse polynomial root finding. \\
	
	As it turns out, these explicit expressions also yield connections to intrinsic geometric and combinatorial information on gamma polynomials pertaining to simplicial complexes and posets. Combining the decomposition with work of Hetyei on Tchebyshev subdivisions \cite{He2}, the inverted Chebyshev basis expansion implies that there is a close connection between the construction of the gamma polynomial out of a reciprocal polynomial and $f$-vectors of simplicial complexes resulting from successive edge subdivisions that transform the type A Coxeter complex to the type B Coxeter complex (Part 1 and Part 2 of Corollary \ref{gamchebdiv}). \\

	Along the way, we find that the gamma vector gives the class in $K_0(\Rep(\mathfrak{sl}(2, \mathbb{C})))$ of the $\mathfrak{sl}(2, \mathbb{C})$-representation with multiplicities of the irreducible representations determined by the (scaled) $h_i$ in reverse order (Part 3 of Corollary \ref{gamchebdiv}). This involves taking a closer look at linear transformations $\mathbb{R}[x] \longrightarrow \mathbb{R}[x]$ and specializations to $SL_2(\mathbb{C})$-representations. We note that it was previously known that the matrix taking a gamma vector to the first $\lfloor \frac{d}{2} \rfloor + 1$ indices of the $h$-vector is totally nonnegative by a Lindstr\"om--Gessel--Viennot argument (see p. 1377 of \cite{NPT} and Section \ref{addrec}). More specifically, we show that the latter gives rise to multiplication-compatible bijections with palindromic/reciprocal unimodal polynomials and Lefschetz-type maps via induced $\mathfrak{sl}_2(\mathbb{C})$-representations (applying work of Almkvist \cite{Alm1}) and study when it is possible to ``reverse-engineer'' an $F$-polynomial structure (Corollary \ref{gamchebdiv}). \\
	
	\color{black}
	Before giving the connection to (topological generalization of) permutation statistics, we give some context on combinatorial interpretations related to the gamma vector. In many examples giving combinatorial interpretations for the gamma vector, the input $h$-vector counts some descent statistic and the resulting gamma vector somehow involves descents and/or filters out double descents. For example, this includes examples ranging from the Eulerian polynomials originally considered by Foata--Sch\"utzenberger to Chow rings of matroids (recent work of Stump \cite{Stu} in general and a combination of Postnikov--Reiner--Williams \cite{PRW} and Ardila--Reiner--Williams \cite{ARW} for those associated to certain hyperplane arrangements constructed out of root systems), and $h$-vectors of nestohedra (Postnikov--Reiner--Williams \cite{PRW}). \\
	
	More specifically: Applying earlier work of Hetyei on Tchebyshev posets \cite{He1}, lifting the Chebyshev polynomials in the inverted basis decomposition to $ce$-indices (a modification of the $cd$-indices) of posets giving triangulations of cross polytopes also gives connections to descent statistics involving edge labelings of maximal chains on posets (Corollary \ref{gamtopdes}). As noted by Hetyei \cite{He1}, this is analogous to the use of Andr\'e permutations (permutations without double descents and an additional property) and their signed versions used by Purtill \cite{Pur} to study cd-indices of Boolean lattices and cubical lattices respectively (Part 1 of Remark \ref{permcrossconn}). It is interesting that we see this without any assumptions on the initial polynomial apart from the reciprocal polynomial assumption $h_k = h_{d - k}$. Apart from this, cross polytopes and repeated edge subdivisions relevant to settings where gamma vectors often studied in relation to flag and balanced simplicial complexes (some of which are listed in Part 2 of Remark \ref{permcrossconn}). \\
	
	Finally, we end by noting that all of these results have analogues for Chebyshev polynomials of the second kind after taking derivatives and point out that they connect the objects that we studied with Hopf algebras and quasisymmetric functions (Remark \ref{secondchebver}).

	\section*{Acknowledgments}

	 The idea to look at Chebyshev polynomials came to mind while listening to a talk about Chebyshev varieties \cite{BAMT} by Chiara Meroni at MEGA 2024 and thinking about gamma vector-like computations for bases of reciprocal polynomials. I'd like to thank the participants and organizers for an enjoyable conference.

	\section{Inverted Chebyshev expansions} \label{invchebexpsect}

	In this section, we show that gamma polynomials $\gamma(u)$ of palindromic/reciprocal polynomials $h(u)$ are closely related to inverted expansions as linear combinations of Chebyshev polynomials (Theorem \ref{gamchebinv}). We were led to this by  studying gamma vector-like invariants for the usual basis $x^k + x^{d - k}$ of the reciprocal/palindromic polynomials of degree $\le d$ and recursions they saisfy when $d$ is even (Proposition \ref{modgambasis}). This is the starting point for our analysis of combinatorial structures ``intrinsically'' related to the gamma vector in later sections. As an initial consequence, we obtain a counterpart of a recent result of Bel-Afia-Meroni-Telen \cite{BAMT} on hyperbolicity of Chebyshev curves for higher dimensions (Corollary \ref{cheblinreal}). Note that Chebyshev varieties share key properties with toric varieties and serve as a counterpart of them in sparse polynomial root finding (see \cite{BAMT}). \\

	Given an input polynomial $h(u)$, one of the formulas in our previous work \cite{Pgam} was of the form \[ \gamma(u) = (J \circ \widetilde{C})(u) \] with \[ J(u) = \frac{h(u)}{(u + 1)^d} \] and $\widetilde{C}(u) \coloneq C(u) - 1$, where $C(u)$ denotes the generating function of the Catalan numbers.  \\
	
	After making the substitution, we have that \[ \gamma(u) = \frac{h(\widetilde{C}(u))}{C(u)^d} \Longrightarrow h(\widetilde{C}(u)) = \gamma(u) C(u)^d. \] 
	
	Suppose that $d$ is even. Since \[ \frac{\widetilde{C}(u)}{(\widetilde{C}(u) + 1)^2} = u \Longrightarrow \frac{C(u) - 1}{C(u)^2} = u \Longrightarrow C(u)^d = \left( \frac{C(u) - 1}{u} \right)^{\frac{d}{2}}, \] this is equivalent to \[  h(\widetilde{C}(u)) = \gamma(u) \left( \frac{\widetilde{C}(u)}{u} \right)^{\frac{d}{2}} \Longrightarrow u^{\frac{d}{2}} h(\widetilde{C}(u)) = \gamma(u) \widetilde{C}(u)^{\frac{d}{2}}.  \]

	Assume that $h(t)$ is a reciprocal polynomial (i.e. $h_k = h_{d - k}$ when $\deg h = d$). We can consider from the perspective of the basis elements of the vector space of reciprocal polynomials used to construct a modified version of the gamma vector. We first list the usual definition of the gamma vector and later combine it with an explicit expression involving compositional inverses using the generating function of the Catalan numbers. \\
	
	\begin{defn} (see proof of Proposition 2.1.1 on p. 272 of \cite{Gal}) \\
		Given a reciprocal polynomial $h(t) = h_0 + h_1 t + \ldots + h_d t^d$ of even degree $d$, the gamma vector is given by the coefficients of the unique polynomial $\gamma(u)$ of degree $\le \frac{d}{2}$ such that \[ h(t) = (1 + t)^d \gamma \left( \frac{t}{(1 + t)^2} \right).   \] Alternatively, $(\gamma_0, \ldots, \gamma_{\frac{d}{2}})$ consists of the unique $\gamma_i$ such that \[ h(t) = \sum_{i = 0}^{ \frac{d}{2} } \gamma_i x^i (1 + x)^{d - 2i}. \]

	\end{defn}
	
	\color{black}
	
	More specifically, substituting in $x = \widetilde{C}(u)$ into $x^i (x + 1)^{d - 2i}$ yields $u^i C(u)^d$ since $\frac{\widetilde{C}(u)}{(\widetilde{C}(u) + 1)^2} = u$. When $h$ is a reciprocal polynomial, it is a linear combination of terms of the form $x^k + x^{d - k}$ for $0 \le k \le \frac{d}{2}$ and substituting in $x = \widetilde{C}(u)$ gives $\widetilde{C}(u)^k + \widetilde{C}(u)^{d - k}$. In this language, the $\gamma_i$ associated to $x^k + x^{d - k}$ yield the polynomial $\widetilde{\gamma}(u)$ (which depends on $k$) such that $\widetilde{C}(u)^k + \widetilde{C}(u)^{d - k} = \widetilde{\gamma}(u) C(u)^d$. We would like to look at the existence of $\widetilde{\gamma}_i$ more explicitly and consider what the polynomial looks like in the case that $d$ is even, which is an assumption made in the usual setting gamma vectors are considered. The latter case will give use more information on what the modified gamma vector-like coefficients associated to the the ``basis elements'' $x^k + x^{d - k}$ generating the vector space of reciprocal polynomials. Note that the gamma vector itself is equal to the linear combination of these modified gamma vector-like coefficients using the same coefficients as the corresponding $x^k + x^{d - k}$.  \\

	\begin{prop} \label{modgambasis}
		Suppose that $d$ is even. For any $0 \le k \le \frac{d}{2}$, we have that \[ x^k + x^{d - k} = \sum_{i \in I} \widetilde{\gamma}_i x^i (x + 1)^{d - 2i} \] for some numbers $\widetilde{\gamma}_i$ indexed over a finite set of nonnegative integers $I$ (depending on $\ell$) and they can be computed recursively. \\
		
		In particular, let $P_r = P_r(z)$ be the reciprocal polynomial $\sum_{i = 0}^{\frac{d}{2}} \widetilde{\gamma}_i^r w^{ \frac{d}{2} - i }$ of the modified $\gamma$-polynomial associated to $x^{\frac{d}{2} - r} + x^{\frac{d}{2} + r}$. Then, $Q_r(z) \coloneq P_r(z + 2)$ satisfies the same recursive relation and has the same initial polynomial at $r = 0$ as the Chebyshev polynomials of the first kind with the variable (say $w$) halved (i.e. replaced by $\frac{w}{2}$). 
	\end{prop}
	
	\begin{proof}
		Since $d$ is even, we can set $\ell = \frac{d}{2} - k$ to get $x^k + x^{d - k} = x^{\frac{d}{2} - \ell} + x^{\frac{d}{2} + \ell}$. After dividing by $x^{\frac{d}{2}}$, the desired expression is given by \[ x^\ell + x^{-\ell} =  \sum_{i \in I} \gamma_i x^{-\left( \frac{d}{2} - i \right)} (x + 1)^{d - 2i} = \sum_{i \in I} \gamma_i \left( \frac{(x + 1)^2}{x} \right)^{\frac{d}{2} - i}. \] 
		
		If $\ell = 1$, we have that $x + x^{-1} = \frac{x^2 + 1}{x} = \frac{(x + 1)^2}{x} - 2$. As for $x^2 + x^{-2} = (x + x^{-1})^2 - 2$, which itself is a polynomial in $\frac{(x + 1)^2}{x}$ with $x^2 + x^{-2} = \left( \frac{(x + 1)^2}{x} - 2 \right)^2 - 2 = \left( \frac{(x + 1)^2}{x} \right)^2 - 4 \left( \frac{(x + 1)^2}{x} \right) + 4 - 2$. The same recursion applies for exponents of the form $\ell = 2^m$. \\
		
		In general, we can use the expression \[ (w + w^{-1})(w^\ell + w^{-\ell}) = (w^{\ell - 1} + w^{-(\ell - 1)}) + (w^{\ell + 1} + w^{-(\ell + 1)}) \] to obtain a recursion. Writing $P_r$ for the polynomial such that for $w^r + w^{-r} = P_r \left( \frac{(w + 1)^2}{w} \right)$, we have that 
		
		\begin{align*}
			P_r \left( \frac{(w + 1)^2}{w} \right)  P_1 \left( \frac{(w + 1)^2}{w} \right) &= (w^{r + 1} + w^{-(r + 1)}) + P_{r - 1} \left( \frac{(w + 1)^2}{w} \right) \\
			\Longrightarrow w^{r + 1} + w^{-(r + 1)} &= P_r \left( \frac{(w + 1)^2}{w} \right)  P_1 \left( \frac{(w + 1)^2}{w} \right) - P_{r - 1} \left( \frac{(w + 1)^2}{w} \right) \\
			\Longrightarrow P_{r + 1}(z) &= P_r(z) P_1(z) - P_{r - 1}(z) \\
			&= (z - 2) P_r(z) - P_{r - 1}(z),
		\end{align*}
		
		which is a polynomial when $P_r(z)$ and $P_{r - 1}(z)$ are. Replacing $z$ by $w \coloneq z + 2$ yields the recursion for the Chebyshev polynomials of the first kind with the variable $w$ replaced by $\frac{w}{2}$ (see p. 8 of \cite{AGSMMS}).

	\end{proof}
	
	When we substitute in $x = \widetilde{C}(u)$, we can get some more explicit information on what the coefficients are and what they count.

	\begin{cor} 
		Suppose that $d$ is even. Let $C(u)$ be the generating function of the Catalan numbers and $\widetilde{C}(u) = C(u) - 1$. Given $0 \le \ell \le \frac{d}{2}$, there is a polynomial $Q(u)$ of degree $\le \ell$ that \[ \widetilde{C}(u)^{ \frac{d}{2} - \ell } + \widetilde{C}(u)^{\frac{d}{2} + \ell} = [u^{\frac{d}{2} - \ell} Q(u)] C(u)^d. \] 
		
		Equivalently, any $0 \le k \le \frac{d}{2}$ has a polynomial $Q(u)$ of degree $\le \frac{d}{2} - k$ such that \[ \widetilde{C}(u)^k + \widetilde{C}(u)^{d - k} = [u^k Q(u)] C(u)^d. \]
		
	\end{cor}

	\begin{proof}

		We now look at details on what the polynomial $Q(u)$ in the ``usual'' initial conditions to get some information about the behavior of the gamma vector on basis elements of reciprocal polynomials. \\
		
		The case $\ell = 0$ follows immediately from taking the $\frac{d}{2}^{\text{th}}$ power of each side of the identity $\widetilde{C}(u) = u C(u)^2$. For the general case, we  start by factoring out $\widetilde{C}(u)^{\frac{d}{2}} = u^{\frac{d}{2}} C(u)^d$. Afterwards, we use the following property:
		
		\begin{lem}
			Given $0 \le \ell \le \frac{d}{2}$, there is a polynomial $P(u)$ of degree $\le \ell$ such that \[ \widetilde{C}(u)^{-\ell} + \widetilde{C}(u)^\ell = \frac{P(u)}{u^\ell}. \]
		\end{lem}
		
		\begin{proof}
			After making the substitutions $z = \frac{(w + 1)^2}{w}$ and $w = \widetilde{C}(u)$  into the identities at the end of the proof of Proposition \ref{modgambasis}, we end up with \[ P_{r + 1} \left( \frac{1}{u} \right) = \left( \frac{1}{u} - 2 \right) P_r \left( \frac{1}{u} \right) - P_{r - 1} \left( \frac{1}{u} \right). \] If the statement holds for $P_r$ and $P_{r - 1}$, this identity implies that it also holds for $P_{r + 1}$. 
			
		\end{proof}

	\end{proof}
	
	These steps can be used to recover the following result on connections between palindromic polynomials and Chebyshev expansions (which is known although the situation with sources is complicated -- see \cite{AGSMMS} for further details). \\

	\begin{prop} (Proposition 3.2 on p. 11 of \cite{AGSMMS}) \label{recipcheb} \\
		Let $f$ be a palindromic polynomial with complex coefficients: 
		\[ f(t) = \sum_{k = 0}^{2n} a_k t^k \] where $a_{2n - k} = a_k$ for every $k$ in $\{ 0, \ldots, n \}$. Define a univariate polynomial $g$ by \[ g(u) \coloneq \sum_{j = 0}^n a_{n - j} T_j^{\textbf{monic}}(u) = a_n + \sum_{j = 1}^n 2 a_{n - j} T_j \left( \frac{u}{2} \right). \]
		
		Then \[ f(t) = t^n g(t + t^{-1}). \]
	\end{prop}
	
	\begin{rem} \label{oddeg}
		In some sense, the case of odd degree reciprocal/palindromic polynomials reduces to the even degree case since odd degree reciprocal polynomials are always multiples of $1 + t$ (Proposition 3.4 on p. 11 of \cite{AGSMMS}). We can divide by $1 + t$ and only consider the even degree case or expand terms and use identities involving Chebyshev polynomials. \\
	\end{rem}
	
	\begin{thm} \textbf{(Gamma vectors and inverted Chebyshev expansions)} \label{gamchebinv} \\
		Let $h(t) = h_0 + h_1 t + \ldots + h_{d - 1} t^{d - 1} + h_d t^d$ be a reciprocal polynomial satisfying the relations $h_k = h_{d - k}$ for $0 \le k \le \frac{d}{2}$. Then, the gamma-polynomial associated to $h(t)$ is equal to \[ \gamma(u) = u^{ \frac{d}{2} } g \left( \frac{1}{u} - 2 \right), \] where \[ g(u) \coloneq h_{ \frac{d}{2} } + 2 \sum_{ j = 1}^{ \frac{d}{2} } h_{ \frac{d}{2} - j } T_j \left( \frac{u}{2} \right). \] This is a sort of ``inverted Chebyshev basis expansion''. \\
		
		Alternatively, we can rewrite this as \[ u^{ \frac{d}{2}  }  \gamma \left( \frac{1}{u} \right)  = g(u - 2) \] or \[ (u + 2)^{ \frac{d}{2} } \gamma \left( \frac{1}{u + 2} \right) = g(u). \]
	\end{thm}
	
	\begin{proof}
		Recall that $\gamma(u) = \frac{h(\widetilde{C}(u))}{C(u)^d}$, where $C(u)$ is the generating function of the Catalan numbers and $\widetilde{C}(u) \coloneq C(u) - 1$. Substituting $t = \widetilde{C}(u)$ into Proposition \ref{recipcheb}, we have that 
		
		\begin{align*}
			h(\widetilde{C}(u)) &= \widetilde{C}(u)^{ \frac{d}{2} } g (\widetilde{C}(u) + \widetilde{C}(u)^{-1}) \\
			&= u^{ \frac{d}{2} } C(u)^d g (\widetilde{C}(u) + \widetilde{C}(u)^{-1}) \\
			\Longrightarrow \gamma(u) &= \frac{h(\widetilde{C}(u))}{C(u)^d} \\
			&= u^{ \frac{d}{2} } g (\widetilde{C}(u) + \widetilde{C}(u)^{-1})
		\end{align*}
		
		since $\widetilde{C}(u) = u C(u)^2$. The expression inside $g$ can be simplified using $\frac{\widetilde{C}(u)}{( \widetilde{C}(u) + 1 )^2} = u$ and we have
		
		\begin{align*}
			\gamma(u) &= u^{ \frac{d}{2} } g (\widetilde{C}(u) + \widetilde{C}(u)^{-1}) \\
			&= u^{ \frac{d}{2} } g \left( \frac{\widetilde{C}(u)^2 + 1}{\widetilde{C}(u)} \right) \\
			&= u^{ \frac{d}{2} } g \left( \frac{\widetilde{C}(u)^2 + 2 \widetilde{C}(u) + 1}{\widetilde{C}(u)} - \frac{2 \widetilde{C}(u) }{\widetilde{C}(u) } \right) \\
			&= u^{ \frac{d}{2} } g \left( \frac{1}{u} - 2 \right).
		\end{align*} 
		
	\end{proof}
	
	As an immediate consequence, we obtain a higher dimensional counterpart to the following recent result of Bel-Afia--Meroni--Telen \cite{BAMT} on Chebyshev varieties, which share key properties with toric varieties and serve as a counterpart of them in the context of sparse polynomial root finding:
	
	\begin{defn} (Bel-Afia--Meroni--Telen, p. 6 of \cite{BAMT}) \\
		Let $A = (a, b) \in \mathbb{N}^2$ and consider the curve parametrized by the map $T_{a, b} : \mathbb{C} \longrightarrow \mathbb{C}^2$ with $T_{a, b}(t) = (T_a(t), T_b(t))$. The Zariski closure of this image is denoted by $\mathcal{X}_{a, b, T}$. \\ \\
	\end{defn}
	
	\begin{thm} (Bel-Afia--Meroni--Telen, Theorem 3.5 on p. 7 of \cite{BAMT}) \label{chebcurvhyp} \\
		For $a \in \mathbb{N} \setminus 0$, the Chebyshev plane curve $\mathcal{X}_{a, a + 1, T}$ is hyperbolic with respect to the origin. That is, any line passing through the origin intersects the curve $\mathcal{X}_{a, a + 1, T}$ in $\deg(\mathcal{X}_{a, a + 1, T}) = a + 1$ real points, counting multiplicities. 
	\end{thm}
	
	The idea is to combine our inverted Chebyshev expansion with known facts about the relationship between real-rootedness of the gamma polynomial and that of the $h$-polynomial. \\
	
	\begin{cor} \label{cheblinreal}
		Let $A(x) = a_0 + a_1 T_1(x) + \ldots + a_r T_r(x)$ be a linear combination of Chebyshev polynomials of the first kind. Then, $A(x)$ has real roots if and only if the reciprocal polynomial of degree $2r$ with coefficients of terms of degree $\le r$ given by the tuple $(a_r, \frac{a_{r - 1}}{2}, \ldots, \frac{a_0}{2})$ is real-rooted. \\
	\end{cor}

	\begin{proof}
		We go backwards from our inverted Chebyshev basis expansion to the shifted reciprocal polynomial of the gamma polynomial. Our starting point is essentially $g(2u)$ from Theorem \ref{gamchebinv} with $\frac{d}{2} = r$ and $a_j$ replacing $h_{ \frac{d}{2} - j }$. This is equal to $(2u + 2)^r \gamma \left( \frac{1}{2u + 2} \right)$, where $\gamma(x)$ is the gamma polynomial associated to the reciprocal polynomial of degree $2r$ with terms of degree $\le r$ given by the tuple $(a_r, \frac{a_{r - 1}}{2}, \ldots, \frac{a_0}{2})$. It is real-rooted if and only if $x^r \gamma \left( \frac{1}{x} \right) $ is real-rooted. This follows from a translation of the variable by $-1$ and division of the new variable by $2$. We reduce to the case where the roots under consideration using the following: \\
		
		\begin{lem}
			A polynomial $Q(x)$ of degree $N$ is real-rooted if and only if its reciprocal polynomial $x^N Q \left( \frac{1}{x} \right)$ is real-rooted.
		\end{lem}
		
		\begin{proof}
			Given a polynomial $Q(x)$, we can write $Q(x) = x^k P(x)$ for some $x \nmid P(x)$ and $k \ge 0$. Let $d = \deg P$. Then, we have that the reciprocal polynomial of $Q(x)$ is given by 
			
			\begin{align*}
				x^{d + k} Q \left( \frac{1}{x} \right) &= x^{d + k} \cdot \frac{1}{x^k} P \left( \frac{1}{x} \right) \\
				&= x^d P \left( \frac{1}{x} \right).
			\end{align*}
			
			Since $x \nmid P(x)$, 0 is \emph{not} a root of $P$. This means that the roots of $x^d P \left( \frac{1}{x} \right)$ are $0$ and reciprocals $\frac{1}{\alpha}$ for the roots $\alpha$ of $P$ (which are nonzero). This implies that $x^{d + k} Q \left( \frac{1}{x} \right)$ is real-rooted if and only $P(x)$ is real-rooted. Thus, $x^{d + k} Q \left( \frac{1}{x} \right)$ is real-rooted if and only if $Q(x)$ is real-rooted. \\
		\end{proof}

		This implies that $x^r \gamma \left( \frac{1}{x} \right)$ is real-rooted if and only if $\gamma(x)$ is real-rooted. We now apply the following known result on the relationship between real-rootedness of the gamma polynomial and that of the $h$-polynomial. \\
		
		\begin{obs} (Observation 4.2 on p. 82 of \cite{Pet}) \\
			If $h$ has palindromic coefficients, then $h(t)$ is real-rooted if and only if $\gamma(h; t)$ has only real roots. Moreover, if the coefficients of $h(t)$ are nonnegative, then all the roots of $h$ are nonnegative and $\gamma(h; t)$ has nonnegative coefficients as well. Thus if $h(t)$ is nonnegative, real-rooted, and palindromic, then it is unimodal. \\
		\end{obs}
		
	\end{proof}
	
	\section{Compatibility between Chebyshev polynomials, type $A$ to $B$ subdivisions, and other linear transformations $\mathbb{R}[x] \longrightarrow \mathbb{R}[x]$}
	
	\subsection{Type A to type B subdivisions} \label{typeatob}

	Returning to the geometric combinatorics setting, recall that we are mainly considering $h$-vectors with $h_i \ge 0$ and $h_i = h_{d - i}$ (Dehn--Sommerville relations). Making the substitution $z \mapsto 2z + 2$, we can translate information on the reciprocal polynomial $z^{ \frac{d}{2} } \gamma(z^{-1})$ of $\gamma(z)$ with $z$ replaced by $2z + 2$ into information on nonnegative linear combinations of Chebyshev polynomials of the first kind. In addition, we use earlier results of Hetyei \cite{He2} Tchebyshev triangulations and linear combinations of Chebyshev polynomials of the first or second kind: \\

	\begin{defn} (Hetyei, Definition 1.1 on p. 571 and Definition 2.1 on p. 574 of \cite{He2}, Definition 2 on p. 921 of \cite{He3}) \label{tchebdef} \\
		\begin{enumerate}
			\item Given a polynomial $F(x) = a_n x^n + \ldots + a_1 x + a_0 \in \mathbb{R}[x]$, the \textbf{Tchebyshev transforms} $T(F)(x)$ and $U(F)(x)$ \textbf{of the first and second kind } of $F(x)$ are given by 
			\[ T(F)(x) \coloneq a_n T_n(x) + \ldots + a_1 T_1(x) + a_0 \]
			and \[ U(F)(x) \coloneq a_n U_{n - 1}(x) + \ldots + a_1 U_0(x), \]
			
			where $T_m(x)$ is Tchebyshev polynomial of the first kind, determined by \[ T_m(\cos(\alpha)) = \cos(m \cdot \alpha) \] and $U_m(x)$ is the Tchebyshev polynomial of the second kind for all $m \ge 0$. \\

			\item  We define a \textbf{Tchebyshev triangulation} $T(\Delta)$ of a simplicial complex $\Delta$ having $m = f_1(\Delta)$ as follows. Number the edges $e_1, \ldots, e_m$ in some order. We subdivide the edge $e_i$ into a path of length 2 by adding the midpoint $w_i$ and we also subdivide all the faces containing $e_i$ by performing a stellar subdivision. Note that the $f$-vector from any initial ordering yields the same $f$-vector (and thus the same $F$-vector) (Theorem 3 on p. 922 of \cite{He3}). \\
		\end{enumerate}
		
	\end{defn}

	The Tchebyshev subdivision is compatible with a modification of the $f$-polynomial of a simplicial complex. \\

	\begin{prop} (Hetyei, Hetyei--Nevo from Proposition 3.3 on p. 578 -- 579 and Proposition 4.4 on p. 580 -- 581 of \cite{He2}, Proposition 5.1 on p. 99 of \cite{HN})  \label{tchebcompat} \\
		Given a simplicial complex $S$, let $F_S(x) \coloneq f_S \left( \frac{x - 1}{2} \right)$. Here, $f_\Delta(t)$ is the usual $f$-polynomial of a simplicial complex $\Delta$.
		
		\begin{enumerate}
			\item The Tchebyshev transform of the $F$-polynomial of a simplicial complex is the $F$-polynomial of its Tchebyshev triangulation: \[ T(F_\Delta)(x) = F_{T(\Delta)}(x). \]
			
			\item  The Tchebyshev transform of the second kind of the $F$-polynomial of a simplicial complex is half of the $F$-polynomial of its Tchebyshev triangulation: \[ U(F_\Delta)(x) = \frac{1}{2} F_{U(\Delta)}(x). \] Note that a Tchebyshev subivision is also the same transformation taking the type $A$ Coxeter complex to the type $B$ Coxeter complex \cite{He3}. Dually, one can consider faces of type A and type B permutohedra (p. 919 of \cite{He3}). \\
		\end{enumerate}
	\end{prop}

	An initial consequence describes when the substitution $z \mapsto 2z + 2$ applied to the reciprocal polynomial of the gamma polynomial is equal to an $F$-polynomial. \\

	\begin{cor} \label{gamchebdiv1} ~\\
		Suppose that $(h_{ \frac{d}{2} }, 2 h_{  \frac{d}{2} - 1 }, \ldots, 2 h_1, 2 h_0 )$ gives the coefficients of $f_\Delta \left( \frac{t - 1}{2} \right)$ for some simplicial complex $\Delta$ of dimension $\frac{d}{2}$. Then, the linear modification $(2u + 2)^{ \frac{d}{2} } \gamma \left( \frac{1}{2u + 2} \right) $ of the reciprocal polynomial of the gamma polynomial is the $F$-polynomial of its Tchebyshev subdivision. Since $F_S(2x + 1) = f_S(x)$, this can also be used to produce the usual $f$-vector.   \\

	\end{cor}

	\begin{rem} 
		Although there has been previous work of Hetyei--Nevo \cite{HN} on (generalizations of) gamma vector and Tchebyshev subdivisions, the focus was on the Tchebyshev subdivisions as an input to the gamma vector-type construction rather than being part of the construction of the gamma vector itself. The type A to B change idea also shows up in Brenti--Welker's analysis of $h$-polynomials of barycentric subdivisions of Boolean cell complexes \cite{BW} and nonnegativity the top gamma vector implied by real-rootedness subdivisions via ``signed permutations'' of type $B$. \\
		
	\end{rem}
	
	\color{red}

	\color{black}
	
	\subsection{Compatibility of Chebyshev polynomials with linear transformations of polynomial spaces and $\mathfrak{sl}_2(\mathbb{C})$-representations} \label{addrec}

	In this subsection, we consider natural compatibility relations between Chebyshev polynomials and linear transformations $\mathbb{R}[x] \longrightarrow \mathbb{R}[x]$ and \emph{specializations} to $\mathfrak{sl}_2(\mathbb{C})$-representations. Note that the latter involve Chebyshev polynomials of the second kind. \\
	
	Given a palindromic/reciprocal polynomial, we started with viewing its gamma polynomial from the perspective of an inverted Chebyshev expansion. Recall from Section \ref{typeatob} that the \emph{linear transformation} $T : \mathbb{R}[x] \longrightarrow \mathbb{R}[x]$ defined by $x^n \mapsto T_n(x)$ is compatible with $f$-vector structures.  In this context, we study what it means to ``reverse engineer'' the $f$-vector structure from the linear transformation perspective on the shifted reciprocal polynomial $(2w + 2)^{ \frac{d}{2} } \gamma \left( \frac{1}{2w + 2} \right)$. Note that the linear transformation taking a gamma vector $\gamma(\Delta)$ to $(h_0(\Delta), h_1(\Delta), \ldots ,h_{ \frac{d}{2} }(\Delta))$ (which involves binomial coefficients) is known to be totally nonnegative by a Lindstr\"om-Gessel-Viennot argument (see p. 1377 of \cite{NPT}). This argument is similar to one in earlier work of Lindstr\"om (Lemma 1 on p. 87 of \cite{Lind}) and combinatorial properties specific to determinants involving binomial coefficients are further explored by Gessel--Viennot (e.g. see Section 1, Section 2, and Section 3 of \cite{GV}). \\

	Apart from simply considering the polynomials as is, we can also assign additional meaning via specializations. This involves compatibility relations between Chebyshev polynomials and linear transformations $\mathbb{R}[x] \longrightarrow \mathbb{R}[x]$ sending $x^n$ to the Chebyshev polynomials of the first or second kind.  More specifically, we look at conditions for ``reverse engineering'' $F$-polynomial structures of inverted Chebyshev expansions and specializing the variable (say $x$) to a generator of the ring of $SL_2(\mathbb{C})$-representations (Corollary \ref{gamchebdiv}). The latter yields connections to multiplication-compatible bijections with palindromic/reciprocal unimodal polynomials and Lefschetz-type maps (after we move to $\mathfrak{sl}_2(\mathbb{C})$-representations) work of Almkvist \cite{Alm1}. Since there is a simple relationship $T_m'(x) = m U_{m - 1}(x)$ expressing derivatives of Chebyshev polynomials of the first kind in terms of those of the second kind, we can relate this to our inverted Chebyshev expansions from earlier (Theorem \ref{gamchebinv}).  \\
	
	\pagebreak 
	
	The discussion above is summarized below:
	
	\color{black} 
	
	\begin{cor} \label{gamchebdiv} ~\\
		\vspace{-3mm}
		
		\begin{enumerate}

			\item Let $A$ be the change of basis matrix used to express a polynomial written in the usual monomial basis of the space of single variable polynomials of a given degree in terms of the basis of Chebyshev polynomials of the first kind (see p. 1 of \cite{BAMT}). Note that $T = A^{-1}$. Let $\widetilde{g}$ be the counterpart of $g$ from Proposition \ref{recipcheb} when terms of initial reciprocal/palindromic polynomial of degree $\le \frac{d}{2}$ have coefficients $(h_{ \frac{d}{2} }, h_{ \frac{d}{2} - 1 }, \ldots, h_1, h_0)$.  Let $v$ and $\widetilde{v}$ be the vector of coefficients of terms of degree $\le \frac{d}{2}$ in $g$ and $\widetilde{g}$. \\
			
			The polynomial $g$ (from Proposition \ref{recipcheb}) for the original reciprocal polynomial $h$ is the $F$-polynomial of the Tchebyshev subdivision $T(\Delta)$ of a simplicial complex $\Delta$ if and only if $A\widetilde{v}$ with the entries in reverse order is the $F$-polynomial of some simplicial complex $\Delta$. \\

			\item  After multiplying by 2, the derivative of the shifted reciprocal polynomial $(2u + 2)^{ \frac{d}{2} } \gamma \left( \frac{1}{2u + 2} \right) $ of the gamma polynomial corresponds to elements of $K_0(\Rep(\mathfrak{sl}(2, \mathbb{C})))$ with the variable $u$ standing for the ``usual'' generator. They come from $SL_2(\mathbb{C})$-representations with constructed out of $j$-dimensional irreducible representations $(1 \le j \le \frac{d}{2}$) $\Sym^j V$ ($V$ being a 2-dimensional vector space) with multiplicities given by $2j h_{ \frac{d}{2} - j }$ for $1 \le j \le \frac{d}{2}$ and $\frac{d}{2} h_{\frac{d}{2}}$ for index $\frac{d}{2}$. We can explicitly compute corresponding palindromic/reciprocal unimodal polynomials in a way that is compatible with multiplication via work of Almkvist \cite{Alm1}. The corresponding/induced $\mathfrak{sl}_2(\mathbb{C})$-representations also yield Lefschetz maps (e.g. see Theorem 14.2.2 on p. 239 -- 240 of \cite{Ara}, p. 5 -- 6 of \cite{Cat}). \\
			
			In addition, the connection to $K_0(\Rep(\mathfrak{sl}(2, \mathbb{C})))$ also connects this to natural correspondences between $\mathfrak{sl}_2(\mathbb{C})$-representations and pairs of unimodal polynomials \cite{Alm1} and categorification-related work \cite{QW}. \\

		\end{enumerate}

	\end{cor}
	
	\begin{rem} ~\\
		
		\begin{enumerate}

			\item 
			\begin{enumerate}

				\item Some further comments on the matrix $A$ and realizability in Corollary \ref{gamchebdiv1} and Part 1: \\
				
				The entries of $A$ (which can be computed using the methods in the proofs of Lemma 3.2 and Lemma 3.4 on p. 250 of \cite{DPS}) have some resemblance to cubical counterparts of the Dehn--Sommerville relations (p. 9 of \cite{Ad}) and other computations involving cubical $h$-vectors. Also, the nonnegativity can be used to construct $f$-vectors when $v$ has nonnegative entries. Then, the nonnegativity of $Av$ implies that scaling by a sufficiently large number (say $\binom{ \frac{d}{2} }{ \frac{d}{4}}$ ) gives the $f$-vector of a simplicial poset (Theorem 2.1 on p. 321 of \cite{Stfh}) to which we can precompose with the invertible matrix yielding subdivision we are considering to get the $f$-vector of a simplicial complex (and the $h$-vector after multiplying by the coordinate change matrix from $f$-vectors to $h$-vectors on the left). To explore what vectors are occur as $F$-vectors of simplicial complexes, one can see how the transformations $t \mapsto 2t + 1$ and $t \mapsto \frac{1 + t}{1 - t}$ interact with the inequalities involved in these realizability results. \\ 
				
				\item  A generalization of Part 1 can be stated as follows: \\
				Let $A$ be the change of basis matrix used to express a polynomial written in the usual monomial basis in terms of the basis of Chebyshev polynomials of the first kind. Suppose that $C$ is any invertible $\left( \frac{d}{2} + 1 \right) \times \left( \frac{d}{2} + 1 \right)$ matrix. Note that $T = A^{-1}$. Let $v$ and $\widetilde{v}$ be the vector of coefficients of terms of degree $\le \frac{d}{2}$ in $g$ and $\widetilde{g}$. Then, we have that \[  v = \left[ A^{-1} C A \right] \widetilde{v} \]  and \[  \widetilde{v} =  \left[ A^{-1} C A \right]  v. \] 
				
				The polynomial $g$ is the $F$-polynomial of the Tchebyshev subdivision $T(\Delta)$ of a simplicial complex $\Delta$ if and only if $CA\widetilde{v}$ is the $F$-polynomial of some simplicial complex $\Delta$. \\
				
				Note that the entries of $A$ are nonnegative. This implies that the entries of $Av$ and $A \widetilde{v}$ are nonnegative when $h_k \ge 0$ for all $0 \le k \le \frac{d}{2}$. \\ 
			\end{enumerate}

			\item  In Part 2, there is a multiplication-compatible bijection between palindromic/reciprocal unimodal polynomials of odd/even degree and $SL_2(\mathbb{C})$-representations that are direct sums of irreducible representation of odd/even dimensions by work of Almkvist \cite{Alm1}. This map is given by $f(t) \mapsto u^{-\deg f} f(\mu^2)$ with $\mu$ being a formal character such that the class of the $n$-dimensional irreducible representation $V_n$ is $\frac{\mu^n - \mu^{-n}}{\mu - \mu^{-1}}$ (p. 289 -- 290 of \cite{Alm1}). Here, $\mu$ is a formal parameter defined by $V_2 = \mu + \mu^{-1}$. Note that $U_{m - 1} \left( \frac{\mu + \mu^{-1}}{2} \right) = \frac{\mu^m - \mu^{-m}}{\mu - \mu^{-1}}$, where $U_{m - 1}$ denotes the $(m - 1)^{\text{st}}$ Chebyshev polynomial of the second kind. In this notation, we have that $V_{n + 1} = U_n(V_2/2)$ and the map $V_2 \mapsto X$ induces an isomorphism between the representation ring of $SL_2(\mathbb{C})$ and $\mathbb{Z}[X]$ (Theorem 2.2 on p. 285 of \cite{Alm1}). While we have mainly worked with Chebyshev polynomials of the \emph{first} kind, the relation $T_m'(x) = m U_{m - 1}(x)$ gives a direct connection to the setting of our inverted Chebyshev expansion  \[ (2u + 2)^{ \frac{d}{2} } \gamma \left( \frac{1}{2u + 2} \right) = h_{ \frac{d}{2} } + 2 \sum_{j = 1}^{ \frac{d}{2} } h_{ \frac{d}{2} - j } T_j(u) \] from Theorem \ref{gamchebinv}. The relation $T_m'(x) = m U_{m - 1}(x)$ is also the source of the derivative and multiplicities mentioned in the statement since \[ \frac{d}{du} (2u + 2)^{ \frac{d}{2} } \gamma \left( \frac{1}{2u + 2} \right) = h_{ \frac{d}{2} } + 2 \sum_{j = 1}^{ \frac{d}{2} } h_{ \frac{d}{2} - j } T_j'(u) =  h_{ \frac{d}{2} } + 2 \sum_{j = 1}^{ \frac{d}{2} } h_{ \frac{d}{2} - j } \cdot j U_{j - 1}(u). \] Suppose that we take $u = V_2/2$ (i.e. $V_2 = 2u$). Recall that $V_2 \mapsto X$. Then, we have \[ \frac{d}{du} (2u + 2)^{ \frac{d}{2} } \gamma \left( \frac{1}{2u + 2} \right) =  h_{ \frac{d}{2} } + 2 \sum_{j = 1}^{ \frac{d}{2} } h_{ \frac{d}{2} - j } \cdot j V_j. \]  The corresponding pair of palindromic/reciprocal unimodal polynomials $F$ and $G$ mapping to sums of terms with designated parities (see Theorem 3.2 on p. 289 of \cite{Alm1}) are of degree $\frac{d}{2}$ and $\frac{d}{2} - 1$. Given $0 \le p, q \le \frac{d}{4}$, the coefficient of $u^p$ in $F(u)$ is

			\[  \frac{d}{2} h_{ \frac{d}{2}  } + 2 \left( \frac{d}{2} - 2 \right) h_{ \frac{d}{2} - 2 } + \ldots + 2 \left( \frac{d}{2} - 2p \right)  h_{ \frac{d}{2} - 2p }.  \]

			\color{black}
			
			and the coefficient of $u^q$ in $G(u)$ is

			\[ 2 \left( \frac{d}{2} - 1 \right) h_{ \frac{d}{2} - 1 } + 2 \left( \frac{d}{2} - 3 \right) h_{ \frac{d}{2} - 3 } + \ldots + 2 \left( \frac{d}{2} - 2q - 1 \right) h_{ \frac{d}{2} - 2q - 1 }. \]
			
		\end{enumerate}
	\end{rem}

	\section{Intrinsic descent statistics and connections to quasisymmetric functions} \label{newchebconn}
	
	In this section, we will discuss new perspectives on the gamma polynomial expression and its connection to the Chebyshev polynomial. For example, a lift of the inverted Chebyshev expansion yielding the translated reciprocal polynomial of the gamma polynomial to a modification of the $cd$-index (the $ce$-index) induces a decomposition into (subdivision of) cross polytopes (Corollary \ref{gamchebdiv}). Each term of the lift enumerated using topological descent statistics. \\

	\subsection{Decompositions into (subdivisions) of cross polytopes, (topological) descents, and Chebyshev polynomials of the second kind}
	
	A combinatorial pattern that is persistent in many gamma positivity examples is a kind of permutation descent count among some set of objects that filters out double descents. For example, this includes Eulerian polynomials, Chow rings of matroids (recent work of Stump \cite{Stu} for general case and earlier work of Postnikov--Reiner--Williams \cite{PRW} and Ardila--Reiner--Williams \cite{ARW} for matroids associated to certain hyperplane arrangements built out of root systems), and $h$-vectors of nestohedra in work of Postnikov--Reiner--Williams \cite{PRW}. Work of Brenti--Welker \cite{BW} on $h$-polynomials of barycentric subdivisions (including the sign of the top gamma vector via real-rootedness) also involves permutation descents. Note that the descent statistics-related expressions for the gamma vector often make use of initial conditions on the input polynomial $h(t)$ (e.g. Theorem 1.1 on p. 210 of \cite{PRW}). \\

	Here, we take a different point of view and consider the structure of the gamma vector construction while only assuming that the $h$-polynomial is a reciprocal polynomial. In particular the inverted Chebyshev expansion shows that an additional ascent-descent statistic is baked into the structure of the gamma vector construction itself. We see this by lifting the Chebyshev polynomials to $ce$-indices (a variant of the $cd$-index -- see p. 495 -- 496 of \cite{He1}) of appropriate posets using earlier work of Hetyei \cite{He1} on Tchebyshev transforms of posets which precedes the work on Tchebyshev subdivisions mentioned earlier. Given a graded poset $P$ and its graded Tchebyshev transform $T(P)$, the order complex $\Delta(T(P) \setminus \{ (\widehat{-1}, \widehat{0}), (\widehat{1}, \widehat{2}) \} )$ triangulates the suspension of $\Delta(P \setminus \{ \widehat{0}, \widehat{1} \} )$ by Theorem 1.5 on p. 573 of \cite{He2}. \\

	\begin{cor} \label{gamtopdes} ~\\
		\vspace{-2mm}
		\begin{enumerate}
			\item Suppose that $\deg h = d$ for some even $d$ and $h(t) = t^d h(t^{-1})$. The $ce$-index of the shifted reciprocal polynomial  $(2u + 2)^d \gamma \left( \frac{1}{2u + 2} \right)$ of the gamma polynomial $\gamma(t)$ is lifted to a sum of $ce$-indices of posets $T_n$ such that the order complexes $\Delta(T_n \setminus \{ \widehat{0}, \widehat{1} \})$ triangulate the boundary of the $n$-dimensional cross polytope. In particular, they are computed by a topological descent statistics (weighted by the $h_i$) for edge labelings of maximal chains of posets. \\
			
			A description of topological descents is given on p. 496 -- 497 of \cite{He1}. More information on topological descents (topologically analogous to a descent in an $EL$ or $CL$ labeling as noted on p. 520 of \cite{BaH}) is given in \cite{BaH} and \cite{Her}. Descents in the context of $CL$ labelings and their connection to permutation descents are on p. 62 -- 64 and in Example 3.4.3 on p. 63 of \cite{Wa}. \\
			
			\item  For $ \left( \frac{2( w + 1)}{w} \right)^{ \frac{d}{2} } \gamma \left( \frac{w}{2(w + 1)} \right)$, a similar statement holds for the difference between specializations of $ce$-indices setting the other variable equal to 1 (counted by a  weighted sum of topological descent statistics) and fixed ``offset functions'' satisfying the same recursion. This comes from considering \emph{reciprocal polynomials} of Chebyshev polynomials of the first kind (see polynomials $R_k(x)$ in Part 2 of the proof). The monomials in $h(t)$ of a given degree play the role of the $h_i$ from Part 1.
		\end{enumerate}
	\end{cor}
	
	\begin{rem} \label{permcrossconn} ~\\
		\vspace{-3mm}
		\begin{enumerate}
			
			\item As mentioned by Hetyei (p. 494 of \cite{He1}) and Ehrenborg--Readdy (p. 79 - 80 of \cite{ER0}), the relation between $ce$-indices of these Tchebyshev posets $T_n$ and topological descents is similar in spirit to work of Purtill \cite{Pur} studying $cd$-indices of the Boolean algebra in terms of Andr\'e permutations introduced by Foata and Sch\"utzenberger (permutations of sets without double descents and an additional property) and a signed version for $cd$-indices of cubical lattices. \\

			\item Keeping in mind the connection between $T_n$ and cross polytopes, we give some context on common connections between common conditions on $h$-polynomials where gamma vectors are studied (e.g. flag or balanced simplicial complexes). By a result of Caviglia--Constantinescu--Varbaro (Corollary 2.3 on p. 474 of \cite{CCV}), the $h$-vector of any flag Cohen--Macaulay simplicial complex is the $h$-vector of a balanced Cohen--Macaulay simplicial complex. Cross polytopes often come up in this context. For example, there is a sequence of cross-flips (a counterpart of bistellar flips respecting the balanced property) transforming a cross polytope into a given balanced $d$-sphere by work of Izmestiev--Klee--Novik (Theorem 3.10 on p. 95 of \cite{IKN}). Speaking of transformations, any two flag simplicial complexes are PL homeomorphic if and only if they can be connected by a sequence of edge subdivisions and their inverses with each step yielding a flag complex (Theorem 1.2 on p. 70 and 77 of \cite{LN}). It would be interesting if we can find a deeper connection to simplicial complexes arising from Tchebyshev subdivisions, which are given by repeated edge subdivisions. \\

		\end{enumerate}
	\end{rem}
	
	\begin{proof}
		\begin{enumerate}
			\item Recall that \[ (u + 2)^{ \frac{d}{2} } \gamma \left( \frac{1}{u + 2} \right) = g(u), \] where \[ g(u) \coloneq h_{ \frac{d}{2} } + 2 \sum_{ j = 1}^{ \frac{d}{2} } h_{ \frac{d}{2} - j } T_j \left( \frac{u}{2} \right). \]

			By Corollary 8.2 on p. 515 of \cite{He1}, the substitution $c \mapsto x, e \mapsto 1$ sends the $ce$-index of Tchebyshev posets $T_n$ (Tchebyshev transforms of ladder posets $L_n$ -- see p. 494, 500 of \cite{He1} and p. 938 of \cite{He3}) into the Chebyshev polynomial of the first kind $T_n(x)$. \\
			
			Before taking this specialization, the $ce$-indices of these posets $T_n$ are enumerated by a ``topological'' ascent-descent statistic in the context of edge labelings on maximal chains of posets while thinking about labels of consecutive edges on a given maximal chain (Corollary 6.3 and start of proof of Theorem 7.1 on p. 511 -- 512 of \cite{He1}). \\
			
			\item The argument is similar to Part 1 except that we use a different recursion for the polynomials involved and a different specialization of the $ce$-index of the posets involved. \\

			The idea is that a certain substitution yields linear combinations of \emph{reciprocal polynomials} $x^m T_m(x^{-1})$ of the Chebyshev polynomials $T_m(x)$, whose defining recursive relation is a specialization of that of $ce$-indices of Tchebyshev posets. Considering initial conditions, it turns out that multiplying by 2 moves the $ce$-indices in question to the actual reciprocal polynomials of the Chebyshev polynomials. \\
			
			Substituting in $u = \frac{w}{2(w + 1)}$, we have that \[ \frac{1}{u} - 2 = \frac{2(w + 1)}{w} - 2 = 2 + \frac{2}{w} - 2 = \frac{2}{w}. \]
			
			This means that

			\begin{align*}
				\gamma \left( \frac{w}{2(w + 1)} \right) &= \left( \frac{w}{2(w + 1)} \right)^{ \frac{d}{2} } g \left( \frac{2}{w} \right) \\
				&= \left( \frac{1}{2(w + 1)} \right)^{ \frac{d}{2} } w^{ \frac{d}{2} } g \left( \frac{2}{w} \right).
			\end{align*}
			
			\color{black}
			Substituting $u = \frac{2}{w}$ into \[  g(u) \coloneq h_{ \frac{d}{2} } + 2 \sum_{ j = 1}^{ \frac{d}{2} } h_{ \frac{d}{2} - j } T_j \left( \frac{u}{2} \right), \] we have that

			\begin{align*}
				g \left( \frac{2}{w} \right) &= h_{ \frac{d}{2} } + 2 \sum_{ j = 1}^{ \frac{d}{2} } h_{ \frac{d}{2} - j } T_j \left( \frac{1}{w} \right) \\
				\Longrightarrow w^{ \frac{d}{2} } g \left( \frac{2}{w} \right) &= h_{ \frac{d}{2} } w^{ \frac{d}{2} } + 2 \sum_{ j = 1}^{ \frac{d}{2} } w^{ \frac{d}{2} } h_{ \frac{d}{2} - j } T_j \left( \frac{1}{w} \right) \\
				&= h_{ \frac{d}{2} } w^{ \frac{d}{2} } + 2 \sum_{ j = 1}^{ \frac{d}{2} } ( h_{ \frac{d}{2} - j } w^{ \frac{d}{2} - j }) \cdot w^j T_j \left( \frac{1}{w} \right).
			\end{align*}
			
			\color{black}
			Thinking about the second term in products contained in the sum in terms of $v = \frac{1}{w}$, they give reciprocal polynomials $v^k T_k(v^{-1})$ of the Chebyshev polynomials of the first kind $T_k(v)$. \\
			
			Recall that \[ T_n(x) = 2x T_{n - 1}(x) - T_{n - 2}(x). \]
			
			Substituting $\frac{1}{x}$ in place of $x$ and multiplying by $x^n$, we have that

			\begin{align*}
				T_n \left( \frac{1}{x} \right) &= \frac{2}{x} T_{n - 1}  \left( \frac{1}{x} \right) - T_{n - 2} \left( \frac{1}{x} \right) \\
				\Longrightarrow x^n T_n \left( \frac{1}{x} \right) &= 2 x^{n - 1} T_{n - 1} \left( \frac{1}{x} \right) - x^2 \cdot x^{n  - 2} T_{n - 2} \left( \frac{1}{x} \right).
			\end{align*}

			Setting $A_m(x) \coloneq x^m T_m(x^{-1})$, this means that \[ A_n(x) = 2 A_{n - 1}(x) - x^2 A_{n - 2}(x). \]
			
			We note that $A_0(x) = 1$ and $A_1(x) = 1$. \\
			
			It turns out these also have a topological descent interpretation using the same posets $T_n$ as in Part 1. This follows from the recursion for the $ce$-index of the posets $T_n$ given by \[ \Psi_{ce}(T_n) = 2c \Psi_{ce}(T_{n - 1}) - e^2 \Psi_{ce}(T_{n - 2}) \] from Proposition 8.1 on p. 514 of \cite{He1}. \\
			
			In addition, we have that $\Psi_{ce}(T_0) = 1$ and $\Psi_{ce}(T_1) = c$ (see Table 1 on p. 514 of \cite{He1} and note that $d$ is unused), which means that keeping $e$ as is and setting $c = 1$ gives the same recurrence as that of $A_m(x) = x^m T_m(x^{-1})$. \\
			
			The difference between $A_m(x)$ and this specialization is then a polynomial $R_k(x)$ satisfying the recursion \[ R_k(x) = R_{k - 1}(x) - x^2 R_{k - 2}(x) \] such that $R_0(x) = 0$ and $R_1(x) = \frac{1}{2}$. This gives the fixed function in question. \\
			
			The remaining part follows from the same results cited as in Part 1.
		\end{enumerate}
	\end{proof}

	We can naturally extend the results above to those on Chebyshev polynomials of the second kind after taking derivatives. This results in connections with lifts to algebraic invariants. \\

	\begin{rem} \textbf{(Counterparts for Chebyshev polynomials of the second kind and quasisymmetric functions)} \label{secondchebver} \\
		Considering derivatives of (modifications of) the gamma vector gives counterparts of the results above for Chebyshev polynomials of the second kind since $T_n'(x) = n U_{n - 1}(x)$ (e.g. see p. 571 of \cite{He2}), where $U_k(x)$ denotes the $k^{\text{th}}$ Chebyshev polynomial of the second kind. Applying work of Ehrenborg--Readdy \cite{ER} (Theorem 10.3 on p. 231 of and Theorem 11.1 on p. 234 of \cite{ER}) gives connections to Hopf algebras and quasisymmetric functions. Note that $U(F_\Delta)(x) = \frac{1}{2} F_{U(\Delta)}(x)$ by a result of Hetyei (Proposition 4.4 on p. 580 of \cite{He2}). \\
	\end{rem}

\end{document}